\documentclass[a4paper, 11pt, reqno]{amsart}
\usepackage{xypic}
\usepackage{amssymb}
\usepackage{amsthm}
\usepackage{bm}
\usepackage{latexsym}
\usepackage{amsmath}
\usepackage{eufrak}
\usepackage{mathrsfs}
\usepackage{amscd}
\usepackage[all]{xy}
\usepackage[usenames]{color}
\usepackage[dvips]{graphicx}
\usepackage{color}
\usepackage{amscd}
\theoremstyle{plain}
\usepackage[usenames,dvipsnames]{pstricks}
\usepackage{epsfig}
\usepackage{pst-grad} % For gradients
\usepackage{pst-plot} % For axes
\usepackage{longtable}

\newtheorem{theorem}{Theorem}[section]
\newtheorem{proposition}[theorem]{Proposition}
\newtheorem{lemma}[theorem]{Lemma}
\newtheorem{corollary}[theorem]{Corollary}

\theoremstyle{definition}

\newcommand{\appsection}[1]{\let\oldthesection\thesection
\renewcommand{\thesection}{Appendix \oldthesection}
\section{#1}\let\thesection\oldthesection}

\newtheorem{definition}[theorem]{Definition}

\theoremstyle{remark}

\newtheorem{remark}[theorem]{Remark}
\newtheorem{example}[theorem]{Example}

\DeclareMathOperator{\spec}{Spec}

\def\Z{{\mathbb{Z}}}
\def\Q{{\mathbb{Q}}}
\def\C{{\mathbb{C}}}
\def\P{{\mathbb{P}}}

\def\K{\textbf{\textit{k}}}

\def\FF{{\bf{F}}}

\def\O{{\mathcal{O}}}
\def\I{{\mathcal{I}}}
\def\L{{\mathcal{L}}}

\def\M{{\mathcal{M}}}

\def\F{{\mathcal{F}}}

\def\N{{\mathcal{N}}}

\def\AR{{\mathcal{A}}}

\def\CC{{\mathcal{C}}}

\newcommand{\pic}{\text{Pic}}

\begin{document}
\bibliographystyle{amsplain}
\title[Construcci\'on]{Chern slopes of surfaces of general type in positive characteristic}
\author{\textrm{Giancarlo Urz\'ua \\ (WITH AN APPENDIX BY RODRIGO CODORNIU)}}

%\date{\today}

\address{Facultad de Matem\'aticas\\ Pontificia Universidad Cat\'olica de Chile\\ Campus San Joaqu\'in\\ Avenida Vicu\~na Mackenna 4860\\ Santiago\\ Chile}
\email{urzua@mat.puc.cl}
\email{racodorn@mat.puc.cl}


\maketitle

\begin{abstract}
Let $\K$ be an algebraically closed field of characteristic $p>0$, and let $C$ be a nonsingular projective curve over $\K$. We prove that for any real number $x \geq 2$, there are minimal surfaces of general type $X$ over $\K$ such that
\begin{itemize}
\item[a)] $c_1^2(X)>0, c_2(X)>0$,

\item[b)] $\pi_1^{\text{\'et}}(X) \simeq \pi_1^{\text{\'et}}(C)$,

%\item[] $\pic^0(X)$ is reduced,

\item[c)] and $c_1^2(X)/c_2(X)$ is arbitrarily close to $x$.
\end{itemize}

In particular, we show density of Chern slopes in the pathological Bogomolov-Miyaoka-Yau interval $(3,\infty)$ for any given $p$. Moreover, we prove that for $C=\P^1$ there exist surfaces $X$ as above with $H^1(X,\O_X)=0$, this is, with Picard scheme equal to a reduced point. In this way, we show that even surfaces with reduced Picard scheme are densely persistent in $[2,\infty)$ for any given $p$.
\end{abstract}

%----------------------------------------------------------------------------------------------------------------------------------------------
\section{Introduction} \label{s0}

Any minimal smooth projective surface of general type $X$ over $\C$ satisfies $c_1^2(X)>0$ (self-intersection of the canonical class), $c_2(X)>0$ (topological Euler characteristic), and the Bogomolov-Miyaoka-Yau (BMY) inequality $$c_1^2(X) \leq 3 c_2(X).$$ Surfaces with $2c_2 < c_1^2 \leq 3c_2$ are precisely the ones with positive topological index. In \cite{RU14}, Roulleau and the author proved that for any real number $x \in [2,3]$ and any integer $g \geq0$, there are minimal surfaces of general type $X$ with $c_1^2(X)/c_2(X)$ arbitrarily close to $x$, and $\pi^{\text{top}}_1(X)$ is isomorphic to the fundamental group of a compact Riemann surface of genus $g$.

The purpose of this article is to extend this geographical result to surfaces over algebraically closed fields of positive characteristic. In this case, it is well-known that the BMY inequality does not hold, and that $c_2$ could be negative or zero; see \cite{SB91}, \cite{Li13}. With respect to fundamental groups, we replace $\pi_1^{\text{top}}$ by the \'etale fundamental group $\pi_1^{\text{\'et}}$; cf. \cite{Gro71,Mu67}.

Let us recall some facts around \'etale simply connected surfaces (i.e. those with trivial $\pi_1^{\text{\'et}}$). Let $X$ be a nonsingular projective surface in positive characteristic. The Riemann-Roch theorem and Serre's duality imply that (see, e.g., \cite[\S3]{P89}) $$c_1^2(X) \leq 5 c_2(X) + 12 h^0(\Omega_X^1).$$ The dimension $h^0(\Omega_X^1)$ may be positive for \'etale simply connected surfaces. This is pathological with respect to characteristic zero, where trivial topological fundamental group implies that $h^0(\Omega_X^1)=0$, and so one easily gets $c_1^2(X) \leq 5 c_2(X)$. On the other hand, if $X$ is minimal and of general type, then $c_1^2(X) >0$. In addition, if $\pi_1^{\text{\'et}}(X)=0$, then $c_2(X)>0$. The latter follows from the classification of surfaces with $c_2(X) \leq 0$ done by Shepherd-Barron in \cite[\S1 and \S2]{SB91}. Therefore, in this case, we can consider again the Chern slope $c_1^2(X)/c_2(X) >0$, and ask: \textit{What is the behavior of Chern slopes of \'etale simply connected surfaces in $[2,\infty)$?}

In general, there has been a persistent interest on surfaces which violate BMY inequality; see for example \cite{Sz79}, \cite[p.142]{BHH87}, \cite{P89}, \cite{SB91}, \cite{Eas08}, \cite{Ja10}, \cite{Jo14}. As far as the author knows, all the examples in the literature have uncontrolled \'etale fundamental group. Our main theorem is
the following.

\begin{theorem}
Let $p$ be a prime number. Let $C$ be a nonsingular projective curve over an algebraically closed field $\K$ of characteristic $p$. Then, for any real number $x \geq 2$, there are minimal surfaces of general type $X$ over $\K$ such that
\begin{itemize}
\item[a)] $c_1^2(X)>0, c_2(X)>0$,

\item[b)]$\pi_1^{\text{\'et}}(X) \simeq \pi_1^{\text{\'et}}(C)$,

\item[c)] and $c_1^2(X)/c_2(X)$ is arbitrarily close to $x$.
\end{itemize}
The surfaces $X$ have a large deformation space.
\end{theorem}

In particular, for any given $p$, \'etale simply connected surfaces of general type densely violate the BMY inequality. We point out that the surfaces $X$ cannot be lifted to even $W_2(\K)$ when $c_1^2/c_2>3$ if $p\ge 3$, and $c_1^2/c_2>4$ if $p=2$, by Langer's result \cite{Lan14}. Also these surfaces are not rigid at all; they have a large deformation space, and the construction involves generic and unbounded discrete parameters. The relation between $C$ and $X$ is that there exists a birational morphism $X' \to X$ such that $X'$ has a fibration over $C$. This fibration is key to compute $\pi_1^{\text{\'et}}(X) \simeq \pi_1^{\text{\'et}}(C)$.

An additional motivation to look at this geographical problem is the expectation to have some version of the BMY inequality for surfaces with reduced $\pic^0$; see the letter from Parshin to Zagier \cite[\S3]{P89} (see also \cite{P87}). In \cite{Ja10}, Jang constructs counterexamples to Parshin's conjecture. He uses a specific subset of Szpiro's examples \cite{Sz79}, which, in particular, are fibrations onto smooth curves of genus bigger than or equal to two. Our second main result is the following.

\begin{theorem}
Let $\K$ be an algebraically closed field of characteristic $p>0$. Then, for any real number $x \geq 2$, there are minimal surfaces of general type $X$ over $\K$ such that
\begin{itemize}
\item[a)] $c_1^2(X)>0, c_2(X)>0$,

\item[b)] $\pi_1^{\text{\'et}}(X)$ is trivial,

\item[c)] $H^1(X,\O_X)=0$,

\item[d)] $c_1^2(X)/c_2(X)$ is arbitrarily close to $x$.
\end{itemize}
\end{theorem}

The main construction is based on the ``asymptotically random covers" developed in \cite{Urz10,Urz11}, which are branched along pathological arrangements of curves in particular ruled surfaces, similar to the arrangements of sections in \cite{Urz11}. This allows to compare asymptotically the log Chern numbers of the arrangements with the Chern numbers of the constructed surfaces. A key for that is a large-scale behavior of Dedekind sums and Hirzebruch-Jung continued fractions due to Girstmair \cite{Gi03,Gi06}, which was already used in \cite{Urz10,Urz11}.

In the present article, an extra outcome of our constructions is the ``asymptotic minimality" theorem, which is stated in general terms in Theorem \ref{asymptoticallyminimal}. In our case, it implies that the Chern numbers of the constructed surfaces and their minimal models are equivalently proportional to the log Chern numbers of the branch divisor. This is important when proving density of Chern slopes. Another highlight is the generalization of a result on the topological fundamental groups of fibered surfaces (see e.g. \cite{Nori83}, \cite{Xiao91}) to the case of \'etale fundamental groups; see the appendix. The approach is based on the analogous result \cite[Exp.X Cor.1.4]{Gro71} for fibrations with reduced fibers. Via a simple argument, we extend it to the case of nonreduced but nonmultiple fibers, which is indeed needed in our approach. Finally, to prove our second main theorem we modify the construction in our first main theorem so that they have a fibration with ordinary fibers. There are some details involved on that, which are fully explained in \S\ref{s5}. The main ingredients are: existence of ordinary cyclic covers of $\P^1$ of high degree, and an irregularity theorem (Theorem \ref{irregularity}) for certain fibrations over $\P^1$ with ordinary fibers.

Let $\K$ be an algebraically closed field of characteristic $p>0$. We now give an overall outline of the construction of the surfaces $X$, in connection to the sections of this article. Let us fix $\alpha >0$, and let $\epsilon >0$. We construct the surfaces $X(=X_q)$ via generically finite morphisms $f_q \colon X_q \to Y$ of degree $q$ coprime to $p$, such that there exists $\lambda' >0$ satisfying $|\lambda' -\alpha| < \epsilon$ and $$\text{lim}_{q \to \infty} \ c_1^2(X_q)/c_2(X_q) = 2+ \lambda'.$$ The expression $2+\lambda'$ is the log Chern slope of a suitable arrangement in a ruled surface $Y$. To achieve this, it is crucial to have a special pathological branch divisor $D \subset Y$ for the morphisms $f$. The pair $(Y,D)$ is built as the minimal simple normal crossings resolution of a pair $(Z,\AR)$, where $Z=\P(\O_C \oplus \I^{-1}) \to C$ is a specific ruled surface over an arbitrary nonsingular projective curve $C$, for certain line bundle $\I$ of degree $e>0$, and where $\AR$ is a specific arrangement of curves made out of sections and fibers; see \S \ref{s1}.

In \S \ref{s2} we review the $q$-th root construction of Esnault-Viehweg \cite{EV92} for surfaces as in \cite{Urz10}.

In \S \ref{s3} we prove ``asymptotic minimality" for a wide range of arrangements. We start recalling the results of Girstmair as stated in \cite{Urz10}.

In \S \ref{s4} we put all together to prove our main theorem. The result on the \'etale fundamental group of a fibration is found in the appendix.

In \S \ref{s5} we prove our second main theorem by slightly modifying the main construction in \S \ref{s4}. Key is the appearance of ordinary fibers \S \ref{s51}, and Theorem \ref{irregularity}.

\subsection*{Acknowledgements}
I am grateful to Michel Raynaud for kindly explaining to me various things on positive characteristic geometry, which was particularly relevant for Section \ref{s5}. The proof of Theorem \ref{asymptoticallyminimal} (asymptotic minimality) was inspired after reading Lemma 2.4 of Julie Rana's e-print \cite{Ra14}. I would like to thank Robert Auffarth, Irene Bouw, Igor Dolgachev, Andrea Fanelli, Junmyeong Jang, Kirti Joshi, Adrian Langer, Christian Liedtke, and Sukhendu Merhotra for many useful e-mail correspondences and conversations. I would also like to thank the referees for detailed and useful reviews. The was part of the Master's thesis of Rodrigo Codorniu \cite{Co14}, supervised by the author, at Pontificia Universidad Cat\'olica de Chile. The author is supported by the FONDECYT regular grant 1150068 funded by the Chilean Government.

\tableofcontents

%Pending:

%--- Global vector fields. Are there any? Use the formulas I have in my thesis for these coverings.

%--- Check Nori's fundamental group. Compare with Lemma 2.5 in ``Weak density of the fundamental group scheme".

%----------------------------------------------------------------------------------------------------------------------------------------------
\section{Pathological arrangements and their log Chern numbers} \label{s1}

In this section we recall some basic definitions and facts around log differentials; cf. \cite{EV92}, \cite[\S2]{Urz10}. We then specialize to the log differentials defined by the specific arrangements of curves we need for the main construction. We will refer constantly to \cite{Urz11}.

Let us fix an algebraically closed field $\K$. The genus of a nonsingular projective curve $C$ defined over $\K$ is $g(C)=\text{dim}_{\K} H^1(C,\O_C)$. Let $Y$ be a nonsingular projective surface over $\K$, and let $D$ be a simple normal crossing divisor in $Y$. This means that $D$ is a reduced divisor with nonsingular irreducible components which intersect transversally.

Let $\Omega_Y^1(\log D)$ be the \textit{sheaf of log differentials} along $D$; cf. \cite[\S 2]{EV92}. It is a locally free $\O_Y$-module of rank two. Notice that $$\bigwedge^2 \Omega_Y^1(\log D) \simeq \O_Y(K_Y + D).$$ In analogy to the Chern invariants of a nonsingular projective surface, the \textit{log Chern classes} of the pair $(Y,D)$ are defined as $\bar{c}_i(Y,D):= c_i({\Omega_Y^1(\log D)}^{\vee})$ for $i= 1,2$, where ${\Omega_Y^1(\log D)}^{\vee}$ is the dual sheaf of $\Omega_Y^1(\log D)$.

The log Chern numbers of $(Y,D)$ are \begin{center} $\bar{c}_1^2(Y,D):= c_1 \big({\Omega_Y^1(\log D)}^{\vee} \big)^2 \ \ \ \
\text{and} \ \ \ \ \bar{c}_2(Y,D):= c_2 \big({\Omega_Y^1(\log D)}^{\vee} \big).$ \end{center} Hence for $D=0$ we recover the Chern numbers of $Y$, which are denoted by $c_1^2(Y)$ and $c_2(Y)$. Sometimes we may drop $Y$ or $(Y,D)$ if the context is clearly understood.

The next proposition is easy to prove; see for example \cite[Prop.4.6]{Urz10}.

\begin{proposition}
Let $D_1,\ldots, D_r$ be the irreducible components in $D$, and let $t_2$ be the number of nodes of $D$. Then, the log Chern numbers of $(Y,D)$ are
$\bar{c}_1^2(Y,D)= c_1^2(Y) - \sum_{i=1}^r D_i^2 + 2 t_2 + 4 \sum_{i=1}^r (g(D_i)-1),$ and $\bar{c}_2(Y,D)= c_2(Y) + t_2 + 2 \sum_{i=1}^r
(g(D_i)-1)$.
\label{logchern}
\end{proposition}

Let $Z$ be a nonsingular projective surface over $\K$, and let $\AR$ be an \textit{arrangement} of nonsingular curves in $Z$, i.e. a finite collection of nonsingular curves. Consider a \textit{log resolution} $\sigma \colon Y \to Z$ of $\AR$, that is, a sequence of blowups over the singular points of $\AR$, and infinitely near singular points of the total transform of $\AR$, such that $D:=\sigma^*(\AR)_{\text{red}}$ is a simple normal crossings divisor. We notice that this differs slightly from the common notion of log resolution, since we do not allow to blow up nonsingular points of total transforms of $\AR$, or points outside of it. The log Chern numbers of the pair $(Z,\AR)$ are defined as the log Chern numbers of the pair $(Y,D)$. By Proposition \ref{logchern}, they are independent of $\sigma$.

Assume that $\K$ has characteristic $p>0$. We are going to construct pairs $(Z,\AR)$ over $\K$ which have pathological Chern numbers in relation to characteristic zero.

We follow \cite[\S 2]{Urz11}. Let us fix an arbitrary nonsingular projective curve $C$ over $\K$ of genus $g=\text{dim}_{\K} H^1(C,\O_C)$. Let $e>0$ be an integer, and let $\I'$ be an invertible sheaf on $C$ of degree $e$. Consider the $\P^1$-bundle $$\pi' \colon \P \big(\O_C \oplus \I'^{-1} \big) \to C .$$ It has a distinguished section ${C'}_0 \subset \P \big(\O_C \oplus \I'^{-1} \big)$ which is the only curve with negative self-intersection. We have ${C'}_0^2=-e$. For instance $\P \big(\O_{\P^1} \oplus \O_{\P^1}(-e) \big)$ is the \textit{Hirzebruch surface} $\bf{F}_e$.

Let $d \geq 3$ be an integer, and let $e$ be sufficiently large. Let us take $d$ general sections $S'_i$ such that $$S'_i \sim {C'}_0 + {\pi'}^*(\I'),$$ $S'_1+ \ldots + S'_d$ has only nodes as singularities, and these nodes project to different points in $C$ under $\pi'$. It is easy to verify that the number of nodes is $\delta:= e \frac{d(d-1)}{2}$.

Let $\rho \colon C \to C$ be the composition of the absolute Frobenius morphism $r$ times \cite[IV \S2]{Hart77}. The degree of $\rho$ is $p^r$. The base change produces the commutative diagram $$ \xymatrix{ \P(\O_{C} \oplus \I^{-1}) \ar[d]_{\pi} \ar[r]^{\varrho} & \P(\O_C \oplus \I'^{-1}) \ar[d]_{\pi'}   \\
C  \ar[r]^{\rho} & C }$$ where $\I:= \rho^*(\I')$. We consider the sections $S_i:=\varrho^*(S'_i)$, which now are pairwise tangent of multiplicity $p^r$ at $e$ distinct points \cite[\S2]{Urz11}. Note that the divisor $S_1+ \ldots +S_d$ has singularities exactly at those highly tangent points for every pair of sections, and they project to distinct $\delta$ points in $C$ under $\pi$.

Let $Z:=\P(\O_{C} \oplus \I^{-1})$. Let $\{F_1,\ldots,F_{\delta} \}$ be the fibers of $\pi$ passing through the singular points of $\sum_{i=1}^d S_i$, and let $S_{d+1}$ be the section of $\pi$ with $S_{d+1}^2=-ep^r$. We define the arrangement (see \cite[Def.7.1]{Urz11}) $$\AR_0 := \{S_1,\ldots,S_d,S_{d+1},F_1,\ldots,F_{\delta}\}.$$

By \cite[Prop.7.1]{Urz11}, we have $\bar{c}_1^2(Z,\AR_0)= (d-1)(2\delta + 4(g-1)-ep^r)+p^r \delta,$ $\bar{c}_2(Z,\AR_0)=(d-1)(2(g-1)+\delta)$, and so $$\frac{\bar{c}_1^2(Z,\AR_0)}{\bar{c}_2(Z,\AR_0)}= 2 + p^r \frac{e(d-2)}{ed(d-1)+4(g-1)}.$$

Let $u,w$ be positive integers. We consider general sections $\{H_1,\ldots,H_{u} \}$ of $\pi$ such that $H_i \sim S_{d+1} + \pi^*(\I)$ for all $i$, and general fibers $\{R_1,\ldots,R_w\}$ of $\pi$. Our key arrangement is defined as $$\AR := \AR_0 \cup \{H_1,\ldots,H_{u}, R_1, \ldots, R_w \}.$$

The singularities of $\AR$ on each $H_i$ are only nodes, and the same is true for the $R_i$. Thus the surfaces for the minimal log resolution of $(Z,\AR_0)$ and $(Z,\AR)$ coincide, and so we can use the log Chern numbers of $(Z,\AR_0)$ and Proposition \ref{logchern} to compute both $\bar{c}_1^2(Z,\AR)=\bar{c}_1^2(Z,\AR_0)-ep^ru + 2 \Upsilon$, and $\bar{c}_2(Z,\AR)=\bar{c}_2(Z,\AR_0)+\Upsilon$, where $$\Upsilon=\frac{u}{2}(u-1)ep^r+u d e p^r + u \delta + 2(g-1) u+w(u+d-2).$$ In particular, we have $$\frac{\bar{c}_1^2(Z,\AR)}{\bar{c}_2(Z,\AR)}= 2 + p^r \frac{\delta -e(d-1)-u e}{(d-1)(2(g-1)+\delta)+\Upsilon}.$$

%\begin{figure}[htbp]
%\includegraphics[width=7cm]{mono.pdf}
%\caption{The arrangement $\AR$ for $C=\P^1$, $p=2$, $e=1$, $d=3$, $r=1$, $u=2$, $w=2$.} \label{f0}
%\end{figure}

%----------------------------------------------------------------------------------------------------------------------------------------------
\section{Review of $q$-th root covers} \label{s2}

This is a review of $q$-th root covers \cite{EV92} in relation to Dedekind sums and Hirzebruch-Jung continued fractions \cite{Urz10}. The following is taken from \cite[\S 1]{Urz10}.

Let $\K$ be an algebraically closed field of characteristic $p\geq 0$, and let $q$ be a prime number with $q \neq p$. Let $Y$ be a nonsingular projective surface over $\K$. Let $D$ be a nonzero effective divisor on $Y$ such that $D_{\text{red}}$ has simple normal crossings. Let $D=\sum_{i=1}^r \nu_i D_i$ be its decomposition into prime divisors. Assume that there exists a line bundle $\L$ on $Y$ satisfying $$ \L^{q} \simeq \O_Y(D) .$$

We construct from the data $(Y,D,q,\L)$ a new nonsingular projective surface $X$ which represents a ``$q$-th root of $D$". Let $s$ be a section
of $\O_Y(D)$, having $D$ as zero set. This section defines a structure of $\O_Y$-algebra on $\bigoplus_{i=0} ^{q-1} \L^{-i}$ by means of the
induced injection $\L^{-q} \simeq \O_Y(-D) \hookrightarrow \O_Y$. The first step in this construction is given by the affine map $f_1 \colon W
\to Y$, where $W:= \spec_Y \Big( \bigoplus_{i=0} ^{q-1} \L^{-i} \Big)$ (as defined in \cite[p.128]{Hart77}).

Because of the multiplicities $\nu_i$'s, the surface $W$ might not be normal. The second step is to consider the normalization $\overline{W}$ of
$W$. Let $f_2 \colon \overline{W} \to Y$ be the composition of $f_1$ with the normalization map of $W$. The surface $\overline{W}$ can be
explicitly described through the following key line bundles. As in \cite{EV92}, we define the line bundles $\L^{(i)}$ on $Y$ as $$ \L^{(i)}:= \L^i \otimes \O_Y\Bigl( - \sum_{j=1}^r \Bigl[\frac{\nu_j \ i}{q}\Bigr] D_j \Bigr)$$ for $i \in{\{0,1,...,q-1 \}}$.

\begin{proposition} $($see \cite[Cor. 3.11]{EV92}$)$
The group $G= \Z/ q \Z$ acts on $\overline{W}$ (so that $\overline{W}/G = Y$), and on ${f_2}_* \O_{\overline{W}}$. Moreover, we have $$ {f_2}_*
\O_{\overline{W}} = \bigoplus_{i=0}^{q-1} {\L^{(i)}}^{-1}.$$ This is the decomposition of ${f_2}_* \O_{\overline{W}}$ into eigenspaces with
respect to this action. \label{quotientdescr}
\end{proposition}

The normalization of $W$ is $\overline{W}=\spec_Y \Big( \bigoplus_{i=0}^{q-1} {\L^{(i)}}^{-1} \Big).$ The
multiplicities $\nu_i$'s can always be considered in the range $0 \leq \nu_i < q$. If we change multiplicities from $\nu_i$ to $\bar{\nu}_i$
such that $\bar{\nu}_i \equiv \nu_i (\text{mod} \ q)$ and $0 \leq \bar{\nu}_i < q$ for all $i$, then the corresponding varieties $\overline{W}$
will be isomorphic over $Y$ (see e.g. \cite[IV]{Urz08}). Therefore, from now on, we assume that $0< \nu_i < q$ for all $i$.

The surface $\overline{W}$ may be singular, but its singularities are rather mild. They are toric surface singularities \cite[Ch.5]{Mus05},
also called Hirzebruch-Jung singularities when the ground field is $\C$ \cite[pp.99-105]{BHPV04}. These singularities exactly occur over the
nodes of $D_{\text{red}}$. Let us assume that $D_i \cap D_j \neq \emptyset$ for some $i\neq j$, and consider a point $P \in D_i \cap D_j$. Then,
the construction above shows that the singularity at $f_2^{-1}(P) \in \overline{W}$ is isomorphic to the singularity of the normalization of
$$\spec \ \K[x,y,z]/(z^q-x^{\nu_i}y^{\nu_j}),$$ where $x$ and $y$ can be seen as local parameters on $Y$ defining $D_i$ and $D_j$ respectively.

One can easily check that it is isomorphic to the affine toric surface defined by the
vectors $(0,1)$ and $(q,-a)$ in $\Z^2$, where $a$ is the unique integer satisfying $\nu_i a+\nu_j \equiv 0 (\text{mod} \ q)$ and $0<a<q$
\cite[Ch.5 pp.5-8]{Mus05}. We denote this isolated singularity by $\frac{1}{q}(1,a)$. The associated \textit{Hirzebruch-Jung continued fraction} is $$ \frac{q}{a} = e_1 - \frac{1}{e_2 - \frac{1}{\ddots - \frac{1}{e_s}}},$$ which we abbreviate as
$\frac{q}{a}=[e_1,...,e_s]$. For each isolated singularity $\frac{1}{q}(1,a)$, we define its \textit{length} as $l(a,q):=s$. There is also a classical \textit{Dedekind sum} \cite{HiZa74} $$s(a,q):= \sum_{i=1}^{q-1} \Bigl(\Bigl(\frac{i}{q}\Bigr)\Bigr)\Bigl(\Bigl(\frac{ia}{q}\Bigr)\Bigr)$$ where, in our case, $((x))=x-[x]-\frac{1}{2}$ for any rational number $x$. Finally, we define the numbers $$c(a,q):= 12 s(a,q)+l(a,q).$$ These numbers will be used to compute the Chern invariants.

It is well known how to resolve $\frac{1}{q}(1,a)$ by means of toric methods, obtaining the same situation as in the complex case (see \cite[Ch.5 pp.5-8]{Mus05}). That is, if $\frac{q}{a}=[e_1,...,e_s]$, then the singularity $\frac{1}{q}(1,a)$ is resolved by a chain of $l(a,q)$ nonsingular rational curves $\{ E_1,\ldots, E_{l(a,q)} \}$, whose self-intersections are $E_i^2=-e_i$.

%\begin{figure}[htbp]
%\includegraphics[width=14cm]{nroot1.pdf}
%\caption{Resolution over a point in $D_i \cap D_j$.} \label{f1}
%\end{figure}

Thus the surface $\overline{W}$ has only singularities of type $\frac{1}{q}(1,a)$ (for various $0<a<q$) over the nodes of $D$. The third and last step is the minimal resolution $f_3 \colon X \to \overline{W}$ of $\overline{W}$. The composition $f_2 \circ f_3$ is denoted by $f \colon X \to Y$. We have $$H^j(X,\O_X) \simeq \bigoplus_{i=0}^{q-1} H^j \big(Y, {\L^{(i)}}^{-1} \big)$$ for all $j$, and if $K_X$ and $K_Y$ are the canonical divisors for $X$ and $Y$ respectively, then we have the $\Q$-numerical equivalence $$ K_X \equiv f^*\Bigl( K_Y + \frac{q-1}{q} \sum_{i=1}^r D_i \Bigr) + \Delta,$$ where $\Delta$ is a $\Q$-divisor supported on the exceptional locus of $f_3$. That formula follows from $K_{\overline{W}} \equiv f_2^* \Bigl( K_Y + \frac{q-1}{q} \sum_{i=1}^r D_i \Bigr)$ using the fact that $q$ is coprime to $p$ and coprime to the $\nu_i$'s, plus local computations as in \cite[IV \S 2]{Hart77}, and from $K_X \equiv f_3^*(K_{\overline{W}}) + \Delta$ since $f_3$ resolves singularities.

Let $\bar{c}_1^2=\bar{c}_1^2(Y,D_{\text{red}})$, $\bar{c}_2=\bar{c}_2(Y,D_{\text{red}})$, $c_1^2:= c_1^2(Y)$, and
$c_2:= c_2(Y)$. Then, by \cite[\S3]{Urz10}, we can express the Chern numbers of $X$ as $$c_1^2(X)= \bar{c}_1^2 q + 2(c_2-\bar{c}_2) + (c_1^2-\bar{c}_1^2+ 2\bar{c}_2- 2 c_2)\frac{1}{q} - \sum_{i<j} c(a_{i,j},q)  D_i.D_j  $$
$$c_2(X)= \bar{c}_2 q + (c_2-\bar{c}_2) + \sum_{i<j} l(a_{i,j},q) D_i.D_j $$ where $0<a_{i,j}<q$ such that $\nu_i a_{i,j} + \nu_j \equiv 0 \ (\text{mod} \ q)$.

%----------------------------------------------------------------------------------------------------------------------------------------------
\section{Asymptotic minimality} \label{s3}

A particular large scale behavior of the numbers $c(a_{i,j},q)$ and $l(a_{i,j},q)$ will be crucial for the main result of this article. (These phenomena have already been used in \cite{Urz10,Urz11}; they are due to Girstmair \cite{Gi03,Gi06}). We now recall them, following the Appendix of \cite{Urz10}.

Let $q$ be a prime number, and let $a$ be an integer with $0<a<q$. Following Girstmair, a \textit{Farey point} (F-point) is a rational number of the form $q \cdot \frac{c}{d}$, $1\leq d \leq \sqrt{q}$, $0\leq c \leq d$, $(c,d)=1$. Fix an arbitrary constant $C>0$. The interval $$ I_{\frac{c}{d}} = \{ x: \ 0\leq x \leq q, \ \Big|x - q \cdot \frac{c}{d} \Big|\leq C
\frac{\sqrt{q}}{d^2} \} $$ is called the \textit{F-neighbourhood} of the point $q \cdot \frac{c}{d}$. We write $\F_d = \bigcup_{c \in \CC}
I_{\frac{c}{d}}$ for the union of all neighbourhoods belonging to F-points of a fixed $d$, where $\CC= \{ c: \ 0\leq c \leq d \ \& \ (c,d)=1
\}$. The \textit{bad set} $\F$ is defined as $$\F = \Big( \bigcup_{1\leq d \leq \sqrt{q}} \F_d \Big) \cap \{1,2,\ldots,q-1 \}.$$ We denote by $\F^c$ the complement of $\F$ in $\{1,2,\ldots,q-1 \}$.

The next theorem is a summary of \cite[Thm.1]{Gi03}, \cite[Thm.2]{Gi03}, and \cite[Thm.9.7]{Urz10}.

\begin{theorem}
Let $q \geq 17$. Then $|\F| \leq C\sqrt{q} \big(\log(q)+2\log(2) \big)$. If $a \in \F^c$, then we have $$l(a,q) \leq \big(2 + \frac{1}{C} \big) \sqrt{q} + 2 \ \ \text{and} \ \ 12|s(a,q)| \leq \big(2+\frac{1}{C} \big)\sqrt{q} + 5.$$
\label{girstmair}
\end{theorem}

Let $(Z,\AR=\sum_i C_i)$ be a pair as in \S \ref{s1}, where $C_i$ are nonsingular curves in $Z$, and let $\sigma \colon Y \to Z$ be the minimal log resolution of $\AR$.

\begin{definition}
We say that the pair $(Z,\AR)$ is \textit{asymptotic} if for any prime number $q$ sufficiently large, there are integers $0<\mu_i<q$ for each curve $C_i$ such that the following hold.
\begin{itemize}
\item[1.] There is a line bundle $\L$ in $Y$ with $\L^q \simeq \O_Y(D)$, where $$D:= \sigma^*\big(\sum_i \mu_i C_i \big)=\sum_i \nu_i D_i,$$ with $0<\nu_i<q$.
\item[2.] If $D_i \cap D_j \neq \emptyset$, then we have that every integer $0<a_{i,j}<q$ defined by $\nu_i a_{i,j}+ \nu_j \equiv 0 \ (\text{mod} \ q)$ belongs to $\F^c$.
\end{itemize}
\label{asym}
\end{definition}

Let$(Z,\AR)$ be an asymptotic pair with $\bar{c}_1^2(Z,\AR)>0$ and $\bar{c}_2(Z,\AR)>0.$ Then we consider the $X_q:=X$ of the $q$-th root cover with data $(Y,q,\L,D)$ of the preceding section. Using the formulas for Chern invariants (at the end of \S \ref{s2}) and Theorem \ref{girstmair}, we conclude that  \[ \lim_{q \to \infty} \frac{c_1^2(X_q)}{c_2(X_q)} = \frac{\bar{c}_1^2(Z,\AR)}{\bar{c}_2(Z,\AR)}. \] Note that we have $c_1^2(X_q) \approx \bar{c}_1^2(Z,\AR)q$ and $c_2(X_q) \approx \bar{c}_2(Z,\AR) q$ for $q>>0$, and so by classification of algebraic surfaces, the nonsingular projective surface $X_q$ is of general type.

The model example of an asymptotic pair is $(\P^2,\AR)$ where $\AR$ is an arrangement of $d$ lines $\{L_1,\ldots,L_d \}$. Then, we assign multiplicities $\mu_i$ to each $L_i$ via partitions of (large) primes $\mu_1 + \ldots + \mu_d=q,$ and we consider the line bundle $\L=\sigma^*(\O_{\P^2}(1))$. It is a consequence of \cite[Thm.6.1]{Urz10} that for large $q$, there are many partitions satisfying Definition \ref{asym} above. It is based on estimating the number of bad partitions among all partitions. It is proved that the ratio of bad over all partitions tends to zero as $q$ approaches infinity, and so random partitions would work with probability tending to $1$. In \cite[Thm.6.1]{Urz10} we prove that we have a large class of asymptotic pairs $(Z,\AR)$, the so called divisible simple crossing arrangements (see \cite[Def.4.1-2]{Urz10}), which generalize the situation of lines via a system of weighted partitions of $q$ \cite[\S5]{Urz10}. In \cite[Thm.8.1]{Urz11}, we expand the set of asymptotic pairs to arrangements of sections of $\P^1$-bundles over a curve, plus certain fibers. For example, the arrangement $\AR_0$ of \S \ref{s1} is such. A weighted partition is also used in \cite[Thm.8.1]{Urz11}, the main difference with \cite[Thm.6.1]{Urz10} is that the allowed singularities for the arrangement are more general.

For all of these asymptotic pairs, the problem is that the surfaces $X$ we construct may not be minimal (See Appendix 2 of \cite{Urz08} for concrete examples), and so the limit of the Chern slope of the minimal model of $X$ may be different than the limit of $\frac{c_1^2(X)}{c_2(X)}$.

\begin{theorem}
Let $(Z,\AR)$ be an asymptotic pair with positive $\bar{c}_1^2(Z,\AR)$ and $\bar{c}_2(Z,\AR)$. Let $f \colon X \to Y$ be the associated (asymptotic) $q$-th root covers. As in \S 3, the morphism $f$ is $f_2 \circ f_3$ where $f_2 \colon \overline{W} \to Y \simeq \overline{W}/(\Z / q\Z)$, and $f_3 \colon X \to \overline{W}$ is the minimal resolution of the cyclic quotient singularities of $\overline{W}$. Assume that $K_{\overline{W}}$ is nef. Let $X'$ be the minimal model of $X$. Then, \[ \lim_{q \to \infty} \frac{c_1^2(X')}{c_2(X')} = \frac{\bar{c}_1^2(Z,\AR)}{\bar{c}_2(Z,\AR)}. \]
\label{asymptoticallyminimal}
\end{theorem}

\begin{proof}
We have $K_X \equiv f_3^*(K_{\overline{W}}) - \Delta$ where $\Delta$ is an effective $\Q$-divisor supported in the exceptional loci of $f_3$. Then
$$K_{\overline{W}}^2 -K_X^2 =-\Delta^2= \sum_{i<j} c(a_{i,j},q) (D_i \cdot D_j) - 2 \frac{q-1}{q} \sum_{i<j} D_i \cdot D_j$$ where as always the numbers $0<a_{i,j}<q$ are defined by $\nu_i a_{i,j}+ \nu_j \equiv 0 \ (\text{mod} \ q)$. Let $t_2:= \sum_{i<j} D_i \cdot D_j$ be the number of nodes of $D_{\text{red}}$.

Let $m:=qn$ where $n$ is a positive integer. Then $$h^0(X,mK_X) \leq h^0(X,f_3^*(mK_{\overline{W}})).$$ Since the singularities are rational, by the theorem of Riemann-Roch for $X$ we have $$ \chi(mK_{\overline{W}})=\chi(f_3^*(m K_{\overline{W}}))=\frac{m(m-1)}{2} K_{\overline{W}}^2 + \chi(\O_X).$$ Now $h^0(X,f_3^*(mK_{\overline{W}}))=h^0({\overline{W}},mK_{\overline{W}})$ implies $$h^0(X,mK_X) \leq h^0({\overline{W}},mK_{\overline{W}}) \leq \frac{m(m-1)}{2} K_{\overline{W}}^2 + \chi(\O_X) + h^1({\overline{W}},mK_{\overline{W}}).$$ On the other hand, by the birational invariance of the plurigenus between nonsingular projective varieties, we have $h^0(X,mK_X)=h^0(X',mK_{X'})$. Let $K_{X'}^2=K_X^2+s$ for some $s \geq 0$. Then by the Riemann-Roch theorem on $X'$ $$ \frac{m(m-1)}{2} (K_X^2+s) + \chi(O_X) - h^2(X',mK_{X'}) \leq h^0(X',mK_{X'}),$$ and so $$s \leq  \sum_{i<j} c(a_{i,j},q) (D_i \cdot D_j) + \frac{2 h^1({\overline{W}},mK_{\overline{W}})}{m(m-1)} + \frac{2 h^2(X',mK_{X'})}{m(m-1)}.$$ Since the pair $(Z,\AR)$ is asymptotic, we have by Theorem \ref{girstmair} (with $C=1$) $$\sum_{i<j} c(a_{i,j},q) (D_i \cdot D_j) \leq \sum_{i<j} \big(12 |s(a_{i,j},q)| + l(a_{i,j},q) \big)(D_i \cdot D_j) \leq (6 \sqrt{q}+7)t_2,$$ and by the asymptotic Riemann-Roch's theorem \cite[App.Thm.2.15]{Ko96} applied to the nef line bundles $mK_{\overline{W}}$ and $mK_{X'}$, we know that $$h^1({\overline{W}},mK_{\overline{W}})=O(m) \ \ \text{and} \ \  h^2(X',mK_{X'})=O(m).$$ Therefore, $ s \leq (6 \sqrt{q}+7)t_2 + \frac{O(m)}{m(m-1)},$ which implies that $\lim_{q \to \infty} \frac{s}{q} =0$.
\end{proof}

Of course, this theorem would not be true without the assumption $K_W$ nef. For instance, the situation of a $(-1)$-curve in $Y$ disjoint from the branch divisor $D_{\text{red}}$ would produce $q$ $(-1)$-curves in $X$.

%----------------------------------------------------------------------------------------------------------------------------------------------
\section{Density of Chern slopes} \label{s4}

Let us fix an algebraically closed field $\K$ of characteristic $p>0$, and let $g\geq 0$ be an integer.

\begin{lemma}
The function $\lambda \colon \Q_{\geq1} \to [0,\infty)$ defined by $$\lambda(x)=\frac{x}{4}+ \frac{1}{4x} -\frac{1}{2}$$ has dense image. In particular, given $\alpha \in [0,\infty)$ and $\epsilon >0$, there is a rational $x=\frac{u}{v}>1$ with $u-v$ odd such that $|\lambda(x)-\alpha|<\epsilon$.
\label{easy}
\end{lemma}
\qed

\begin{proposition}
Let $e>0$ and $w>0$. Given real numbers $\alpha>0$, and $\epsilon>0$, there are integers $r>>0$, $d>>0$, and $u>>0$ such that $$\Big| \frac{p^r(\delta -e(d-1)-u e)}{(d-1)(2(g-1)+\delta)+\Upsilon} - \alpha \Big| < \epsilon,$$ where $\Upsilon=\frac{u}{2}(u-1)ep^r+u d e p^r + u \delta + 2(g-1) u+w(u+d-2)$, and $\delta=e\frac{d(d-1)}{2}$. \label{denseprop}
\end{proposition}

\begin{proof}
Let $v:= 2d-1+u$. Then, we have the expression $$\frac{(d-1)(d-2)-2u}{u(u-1)+2ud} = \lambda(u/v)-\frac{1}{u} -\frac{1}{v} + \frac{3}{4uv}.$$ Take $u,v$ as in Lemma \ref{easy} such that $|\lambda(x)-\alpha|<\frac{\epsilon}{5}$, $\frac{1}{u}<\frac{\epsilon}{5}$, $\frac{1}{v}<\frac{\epsilon}{5}$, and $\frac{3}{4uv}<\frac{\epsilon}{5}$. We can achieve this by multiplying $u$ and $v$ with an arbitrarily large odd integer (thus $u-v$ is odd). In this way, we can solve $2d=v+1-u$ for an integer $d \geq 3$.

We note that for $d$ and $u$ sufficiently large, the denominator $(d-1)(2(g-1)+\delta)+\Upsilon$ will not be zero. Let us fix $u,v$ from the previous paragraph so that this happens.

Finally, let $r>>0$ be an integer such that $$\Big| \frac{p^r(\delta -e(d-1)-u e)}{(d-1)(2(g-1)+\delta)+\Upsilon} - \frac{(d-1)(d-2)-2u}{u(u-1)+2ud} \Big| < \frac{\epsilon}{5}.$$

Then we have $$\Big| \frac{p^r(\delta -e(d-1)-u e)}{(d-1)(2(g-1)+\delta)+\Upsilon} - \alpha \Big| < $$ $$ \Big|\frac{p^r(\delta -e(d-1)-u e)}{(d-1)(2(g-1)+\delta)+\Upsilon} - \frac{(d-1)(d-2)-2u}{u(u-1)+2ud} \Big| + \ \ \ \ \ \ \ \ \ \ \ \ \ \ \ \ \ $$ $$ \ \ \ \ \ \ \ \ \ \ \ \ \ \ \ \ \ \ \ \ \ \ \ \ \ \ \ \ \ \ \ \ \ \ \ \ \ \ \ \ \Big| \frac{(d-1)(d-2)-2u}{u(u-1)+2ud} - \alpha \Big| < $$ $$\frac{\epsilon}{5} + \Big| \lambda(u/v)-\frac{1}{u} -\frac{1}{v} + \frac{3}{4uv}  - \alpha \Big| < \frac{\epsilon}{5} + \Big| \lambda(u/v)- \alpha \Big| + \frac{1}{u} + \frac{1}{v} + \frac{3}{4uv} < \epsilon $$

\end{proof}

Let $\alpha>0$, and let $0<\epsilon<<1$. Let $C$ be a nonsingular projective curve of genus $g \geq 0$, and $P \in C$ some point in $C$. We consider integers $w>0$, $e>0$, and the invertible sheaf $\I'=\O_C(eP)$, so that we can choose an arbitrary number of general sections $\{S'_1, \ldots, S'_d \}$ in $\P(\O_{C} \oplus \I'^{-1})$, where (as in \S \ref{s1}) $S'_i \sim C'_{0} + \pi'^*(\I')$. Let us fix integers $d,r,u>>0$ as in Proposition \ref{denseprop}. As in \S \ref{s1}, let $Z=\P(\O_{C} \oplus \I^{-1})$, and consider the arrangement $$\AR = \{S_1,\ldots,S_d,S_{d+1},F_1,\ldots,F_{\delta},H_1,\ldots, H_u,R_1,\ldots, R_w\}.$$ Let $\pi \colon Z \to C$ be the defining $\P^1$-bundle. We have
$$S_i \sim S_{d+1} + \pi^*(\I) \sim H_j$$ where $S_{d+1}$ is the negative section with $S_{d+1}^2=-ep^r$.

As in \cite[\S8]{Urz11}, let us consider integer solutions $\{0<x_i\}_{i=1}^{d+u}$, $\{0<y_i \}_{i=1}^{\delta+w}$ for the equation \begin{equation} \label{eq} \sum_{i=1}^{d+u} e p^r x_i + \sum_{i=1}^{\delta+w} y_i = q \end{equation} for some prime number $q \neq p$ sufficiently large. Let $x:= q - \sum_{i=1}^{d+u} x_i$. For $q$ large enough, the equation above has indeed
nonnegative solutions, exactly $$ \frac{q^{d+u+\delta+w-1}}{(d+u+\delta+w-1)! (ep^r)^{d+u}} + O(q^{d+u+\delta+w-2}).$$ In this way,
$$\sum_{i=1}^{d} x_i S_i + \sum_{i=1}^{u} x_{d+i} H_i + x S_{d+1} + \sum_{i=1}^{\delta} y_i F_i + \sum_{i=1}^{w} y_{\delta +i} R_i \ \ \ \ \ \ $$ $$ \ \ \ \ \ \ \ \ \ \ \sim q S_{d+1} + \pi^* \Big(  \O_{C}\big( (\sum_{i=1}^{d+u} x_i) ep^r P \big) + \M \Big),$$
for some line bundle $\M$ on $C$ of degree $\sum_{i=1}^{\delta+w} y_i$. Then, since $$\big(\sum_{i=1}^{d+u} x_i \big)ep^r + \deg \M =q,$$ there is an
invertible sheaf $\N$ on $C$ such that $$\sum_{i=1}^{d} x_i S_i + \sum_{i=1}^{u} x_{d+i} H_i+ x S_{d+1} + \sum_{i=1}^{\delta} y_i F_i + \sum_{i=1}^{w} y_{\delta +i} R_i \sim q ( S_{d+1} + \pi^* \N).$$

Let $\sigma \colon Y \to Z$ be the minimal log resolution of $\AR$. In \cite[\S8]{Urz11} we prove that we can assign solutions of the above equation (\ref{eq}) to the corresponding curves in $\AR$, so that the pair $(Z,\AR)$ is asymptotic (as in \S \ref{s3}). Here $\L:= \sigma^*(S_{d+1} + \pi^* \N)$, and $$D:=\sum_j \nu_j D_j:= \sigma^* \Big(\sum_{i=1}^{d} x_i S_i + \sum_{i=1}^{u} x_{d+i} H_i + x S_{d+1} + \sum_{i=1}^{\delta} y_i F_i + \sum_{i=1}^{w} y_{\delta +i} R_i \Big).$$ As before, using $\L^q \simeq \O_Y(D)$, we have the $q$-th root cover $f=f_2 \circ f_3 \colon X \to \overline{W} \to Y$ where $f_3 \colon X \to \overline{W}$ is the resolution of singularities of the normal projective surface $\overline{W}$ constructed in \S \ref{s3} via the finite morphism $f_2 \colon \overline{W} \to Y$.

\begin{proposition}
We have that $K_{\overline{W}}$ is nef.
\label{nef}
\end{proposition}

\begin{proof}
We know that $$qK_{{\overline{W}}} \equiv f_2^*\Big(qK_Y+(q-1) D_{\text{red}} \Big).$$ Let $P_1, \ldots, P_{\delta}$ be the points in $Z$ where pairs of sections in $\{S_1,\ldots,S_d \}$ intersect. Let $G_{i,j}$ be the exceptional curves over $P_i$ for $\sigma$, and let $E_i := \sum_{j=1}^{p^r} j G_{i,j}$. In this way $$K_Y \equiv -2 S_{d+1} -(2-2g+ep^r) F + \sum_{i=1}^{\delta} E_i $$ where $F$ is the general fiber of $\pi \circ \sigma \colon Y \to Z \to C$. We always use the same letter for the strict transform of a curve. Note that $$D_{\text{red}} \equiv \sum_{i=1}^{d+1} S_i + \sum_{i=1}^u H_i + \sum_{i=1}^{\delta} F_i + \sum_{i=1}^{w} R_i + \sum_{i=1}^{\delta} \sum_{j=1}^{p^r} G_{i,j},$$ and so $D_{\text{red}} \equiv \sum_{i=1}^{d+1} S_i +\sum_{i=1}^u H_i + (\delta+w) F,$ and $$\sum_{i=1}^d S_i \equiv d S_{d+1} + d e p^r F - 2 \sum_{i=1}^{\delta} E_i .$$ Hence $qK_Y + (q-1) D_{\text{red}}$ is numerically equivalent to

$$\frac{(q-1)(d-2)-4}{2} S_{d+1} + \frac{q-1}{2} \sum_{i=1}^d S_i + \sum_{i=1}^{\delta} E_i + t F + (q-1) \sum_{i=1}^u H_i$$ where $t:= (\delta+w+\frac{dep^r}{2})(q-1)-q(2-2g+ep^r)>0$ for $q>>0$. In this way, the divisor $qK_{{\overline{W}}}$ has $\Q$-effective support on $\overline{S_i}:=f_2^*(S_i)_{red}$, $\overline{G_{i,j}} := f_2^*(G_{i,j})_{red}$, $\overline{H_i}:= f_2^*(H_i)_{red}$, and $\overline{F}:= f_2^*(F)_{red}$ for all $i,j$. We have the pullback relations $q\overline{S_i}\equiv f_2^*(S_i)$, $q \overline{H_i}\equiv f_1^*(H_i)$, $q\overline{G_{i,j}} \equiv f_2^*(G_{i,j})$, and $\overline{F} \equiv f_2^*(F)$.

To prove that $K_{\overline{W}}$ is nef, it is enough to show that $K_{\overline{W}} \cdot \Gamma \geq 0$ for $\Gamma \in \{ \overline{S_i},\overline{G_{i,j}},\overline {H_i}, \overline{F} \}$. We have $$K_{\overline{W}} \cdot \overline{S_i} = \frac{1}{q^2} f_2^*(qK_Y +(q-1)D_{red}) \cdot f_2^*(S_i) = \frac{1}{q}(qK_Y +(q-1)D_{red}) \cdot S_i$$ and so $K_{\overline{W}} \cdot \overline{S_i} = \frac{1}{q} (\frac{-(q-1)}{2} (d-2)ep^r + ep^r(d-1)+t+(q-1)uep^r)>0$ for $q>>0$ and $1 \leq i \leq d$. We also have $K_{\overline{W}} \cdot \overline{S_{d+1}}=\frac{1}{q}(q(2g-2)+ep^r+(\delta+w)(q-1)) >0$. Similarly $K_{\overline{W}} \cdot \overline{G_{i,j}} = \frac{1}{q} (q-2)$ if $G_{i,j}^2=-1$, or $0$ if $G_{i,j}^2=-2$. Finally $K_{\overline{W}} \cdot \overline{F}= (d-1)q-(d+1)+(q-1)u>0$ for $q>>0$, and $K_{\overline{W}} \cdot \overline{H_i} = \frac{1}{q} (\frac{q-1}{2}dep^r + t + (q-1)uep^r)$.
\end{proof}

Therefore, for $q>>0$ the nonsingular projective surfaces $X$ corresponding to the asymptotic pair $(Z,\AR)$ are asymptotically minimal by Theorem \ref{asymptoticallyminimal}, this is, the minimal models $X'$ of $X$ satisfy \[ \lim_{q \to \infty} \frac{c_1^2(X')}{c_2(X')} = \frac{\bar{c}_1^2(Z,\AR)}{\bar{c}_2(Z,\AR)}. \] Then by Proposition \ref{denseprop} (and formula at the end of \S \ref{s2}), we have that $$\Big|\frac{c_1^2(X')}{c_2(X')}-(2+\alpha) \Big| \leq \epsilon$$ for $q>>0$.

On the other hand, we note that $\pi \circ \sigma \colon Y \to C$ induces a fibration $h \colon X \to C$ where $h= \pi \circ \sigma \circ f$. The strict preimages of the sections $S_i$ are sections of $h$, and $f^*(\sigma^*(R_i))$ is a fiber of $h$ for any $1\leq i \leq \delta$, whose support is a tree of $\P^1$'s. By Corollary \ref{corlemma}, we have $\pi_1^{\text{\'et}}(X) \simeq \pi_1^{\text{\'et}}(C)$. Since the \'etale fundamental group is preserved under the blowup of a nonsingular point, we obtain that $\pi_1^{\text{\'et}}(X') \simeq \pi_1^{\text{\'et}}(C)$.

Therefore, we have a proof of the main theorem.

\begin{theorem}
Let $C$ be a nonsingular projective curve over an algebraically closed field of characteristic $p>0$. Then, for any number $x \geq 2$, there are minimal surfaces of general type $X'$ with
\begin{itemize}
\item[a)] $c_1^2(X')>0, c_2(X')>0$,

\item[b)]$\pi_1^{\text{\'et}}(X') \simeq \pi_1^{\text{\'et}}(C)$,

\item[c)] and $c_1^2(X')/c_2(X')$ arbitrarily close to $x$.

\end{itemize}

The surfaces $X$ have a large deformation space.
\label{main}
\end{theorem}

\begin{remark}
In the construction above, the parameter $w$ is arbitrary. This is the number of general fibers $\{R_1,\ldots,R_w\}$ in $Y$, which give a large number of moduli parameters in the construction of $X$. We also have $u$ general sections $\{ H_1,\ldots,H_{u}\}$, where $u$ is very large.

Another point for considering these curves $H_i,R_j$ is that one could use them to prove that $\pic^0(X')$ is a reduced scheme, for certain curves $C$ of genus $g$ (arbitrary), provided that we have a Kawamata-Viehweg vanishing theorem for rational surfaces as in \cite[Thm.2.64]{KM1998}. We do not show details here about how to use that conjectured vanishing theorem to prove that $\pic^0(X')$ is reduced. Instead, in the next section we take another point of view to prove such result, in the case when $C=\P^1$ and for a modified construction, but in the spirit of Theorem \ref{main}.
\end{remark}

%----------------------------------------------------------------------------------------------------------------------------------------------
\section{Simply-connected density in $[2,\infty)$ with reduced $\pic^0$} \label{s5}

%----------------------------------------------------------------------------------------------------------------------------------------------
\subsection{Ordinary cyclic covers of $\P^1$ of high degree} \label{s51}

Let $C$ be a nonsingular projective curve of genus $g$ over an algebraically closed field $\K$ of characteristic $p>0$. The \textit{p-rank} of $C$ is the integer $0\leq \sigma(C) \leq g$ such that the number of $p$-torsion points of its Jacobian is $p^{\sigma(C)}$.

For any variety $X$ defined over $\K$, let $$F \colon H^1(X,\O_X) \to H^1(X,\O_X)$$ be the $p$-linear morphism induced by the action of the Frobenius morphism. As in \cite[p.38]{Se58} (see also \cite[\S5]{Se58b} or \cite[p.101]{Mum67} or \cite[p.321]{MOagII}), we can decompose $$H^1(X,\O_X)= H^1(X,\O_X)_{ss} \oplus H^1(X,\O_X)_n $$ where $H^1(X,\O_X)_{ss}$ is the semi-simple part (where $F$ is bijective), and $H^1(X,\O_X)_{n}$ is the nilpotent part. If $X=C$, the dimension of $H^1(X,\O_X)_{ss}$ over $\K$ is $\sigma(C)$. Another way to compute the $p$-rank is via the isomorphism $H_{\text{\'et}}^1(C,\Z/p\Z) \simeq (\Z/p\Z)^{\sigma(C)}$. The curve $C$ is said to be \textit{ordinary} if $\sigma(C)=g$.

We are interested in the $p$-rank of curves defined as cyclic covers of $\P^1$ of degree $q$ coprime to $p$. Such cyclic covers are defined by $r$ points $\{P_1, \ldots, P_r \}$ in $\P^1$ together with multiplicities $\textbf{a}:=\{a_1,\ldots,a_r \}$ such that $0 < a_i < q$, $a_1+a_2+\ldots+a_r=ql$ for some integer $l$, and gcd$(a_1,\ldots,a_r,q)=1$. Then, in analogy to \S  \ref{s2} (see e.g. \cite[IV]{Urz08}), there is a cyclic cover $f \colon C \to \P^1$ branched exactly at the points $\{P_1, \ldots, P_r\}$ so that $C=\spec_{\P^1} \Big( \bigoplus_{i=0}^{q-1} {\L^{(i)}}^{-1} \Big),$ where $$\L^{(i)}= \O_{\P^1} \big( il \big) \otimes \O_{\P^1} \Big( -\sum_{j=1}^r \Bigl[\frac{a_j \ i}{q}\Bigr] P_j \Big) \simeq \O_{\P^1} \Big( il - \sum_{j=1}^r \Bigl[\frac{a_j \ i}{q}\Bigr] \Big)$$ for $i \in{\{0,1,\ldots ,q-1 \}}$. We have $$H^1(C,\O_C) \simeq \bigoplus_{i=0}^{q-1} H^1 \big(\P^1, {\L^{(i)}}^{-1} \big),$$ using the decomposition of $f_* \O_C$ as eigenspaces with respect to the action of $\Z/q\Z$. As in \cite[Section 3, (1) and (4)]{B01}, the action of $F$ on $H^1(C,\O_C)$ sends $H^1(\P^1, {\L^{(i)}}^{-1})$ to $H^1(\P^1, {\L^{(ip)}}^{-1})$, and so $$\sigma(C) \leq \sum_{i=1}^{q-1} \text{min} \{ \text{dim}_{\K} H^1 \big(\P^1, {\L^{(ip^k)}}^{-1} \big) \colon k \in (\Z /q\Z)^*  \}=: B(\textbf{a},0)$$

Notice that $\text{dim}_{\K} H^1(\P^1, {\L^{(i)}}^{-1})= il- \sum_{j=1}^r \bigl[\frac{a_j i}{q}\bigr]-1$ by Riemann-Roch theorem.

\begin{example}
Let us take a partition of a large prime number $$a_1+\ldots+a_r=q.$$ Assume that $p$ is a generator of the multiplicative group $(\Z/q\Z)^*$. Then $ \text{dim}_{\K} H^1(\P^1, {\L^{(1)}}^{-1})=0$. On the other hand, there is only one orbit for the $ip^k$, and so $B(\textbf{a},0)=0$. In this way, the $p$-rank of such cyclic covers is always $0$. Notice that by Artin's conjecture on primitive roots, there should exist infinitely many such $q$ for any given prime $p$.

\label{prank0}
\end{example}

By \cite[Lemma 4.6]{B01}, given the data $q,\textbf{a}$, the set of points $\{P_1,\ldots,P_r\}$ for which $\sigma(C)=B(\textbf{a},0)$ form an open set of $\underbrace{\P^1 \times \ldots \times \P^1}_r$. This open set could be empty.

\begin{lemma}
Let $q$ be a positive integer coprime to $p$, and let $\{a_1,\ldots,a_l\}$ be such that $0<a_j<q$, and gcd$(a_1,\ldots,a_l,q)=1$. We consider $$\textbf{a}=\{a_1,q-a_1,a_2,q-a_2,\ldots,a_l,q-a_l \}.$$ Then the set of points $\{P_1,\ldots,P_{2l}\}$ for which the corresponding cyclic covers are ordinary form a nonempty open set of $\underbrace{\P^1 \times \ldots \times \P^1}_r$.
\label{ordinary}
\end{lemma}

\begin{proof}
By \cite[Lemma 4.6]{B01}, we need to prove existence of at least one such ordinary curve. We will use a degeneration argument as in \cite[Prop.7.4]{B01}. Basically, we are going to construct a degeneration of such cyclic covers into an ordinary stable curve. We will use again \S \ref{s2} to explain the construction.

Let us consider lines $\{L_1,\ldots,L_{2l} \}$ in $\P^2$ intersecting at only nodes, and having $l$ colinear nodes. We denote these nodes by $Q_1,\ldots,Q_l$, and the line containing them by $L$.

%\begin{figure}[htbp]
%\includegraphics[width=8cm]{lines.pdf}
%\caption{}
%\label{f2}
%\end{figure}

As in \S \ref{s2}, we consider the data $$\O_{\P^2} (a_1 L_1 + (q-a_1) L_2 + a_2 L_3 + \ldots + a_l L_{2l-1}+ (q-a_l) L_{2l}) \simeq \O_{\P^2}(l)^{\otimes q}$$ Let $\tau \colon Y \to \P^2$ be the blowup of all $Q_i$ and $P \in L$, $P \neq Q_i$ for all $i$. Let $E_i$ be the $(-1)$-curve over $Q_i$ for $1 \leq i \leq l$, and let $E_{l+1}$ be the $(-1)$-curve over $P$. Then $$\Big(\tau^*(\O_{\P^2}(l)) \otimes \O_Y(-\sum_{i=1}^{l+1} E_i)\Big)^{\otimes q} \simeq \O_Y(a_1 L_1 + (q-a_1) L_2 + \ldots + (q-a_l) L_{2l})$$

Let $f \colon X \to Y$ be the corresponding $q$-th root cover branch along $L_1+\ldots+L_{2l}$. Then, the trivial pencil at $P \in \P^2$ induces a fibration $g \colon X \to \P^1$ so that the fiber over $L$ in $X$ is a reduced nodal curve $C_0$ formed by a grid of $ql$ $\P^1$'s ($q$ ``vertical" and $l$ ``horizontal"). Notice also that the general fiber of $g$ is a cyclic cover of $\P^1$ with data $\textbf{a}$. As in \cite[\S7]{B01}, we have that for a general fiber $C$ of $g$, the inequality $\sigma(C) \geq p_a(C_0)$ holds, since all components of $C_0$ are $\P^1$'s. But $p_a(C_0)=(q-1)(l-1)=g(C)$, and so $\sigma(C)=g(C)$, this is, the general fiber of $g$ is ordinary.

%\begin{figure}[htbp]
%\includegraphics[width=6.5cm]{grid.pdf}
%\caption{}
%\label{f3}
%\end{figure}

\end{proof}

Notice that the lemma above shows existence of arbitrary high degree ordinary cyclic covers for any given $\K$ of characteristic $p$.

%----------------------------------------------------------------------------------------------------------------------------------------------
\subsection{Irregularity for certain ordinary fibrations} \label{s52}

\begin{theorem}
Let $X$ be a nonsingular projective surface over an algebraically closed field $\K$ of characteristic $p>0$. Let $h \colon X \to \P^1$ be a fibration with $h_* \O_X = \O_{\P^1}$. Assume that $h$ has a section and a nonsingular ordinary fiber, and that $X$ has no \'etale Galois covers for $\Z / p \Z$.

Then $H^1(X,\O_X)=0$.
\label{irregularity}
\end{theorem}

\begin{proof}

We assume that the genus of the generic fiber is bigger than zero, otherwise we are done. First we notice that, because of the section, every fiber $h^*(y)$ (with $y \in \P^1$) has a reduced component, and so $h^0(h^*(y), \O_{h^*(y)})=1$. By the cohomology and base change theorem \cite[p.290]{Hart77}, we have that $R^1 h_* \O_X$ is a locally free sheaf on $\P^1$, and so $R^1h_* \O_X \simeq \oplus_i \O_{\P^1}(a_i)$ for some $a_i \in \Z$ (see e.g. \cite[V \S 2]{Hart77}). On the other hand, by the Leray spectral sequence, we have $$H^1(X,\O_X) \simeq H^0(\P^1, R^1 h_* \O_X),$$ and so $H^1(X,\O_X) \simeq \bigoplus_i H^0(\P^1, \O_{\P^1}(a_i))$.

Following Serre's approach as above \cite{Se58}, we have the decomposition $H^1(X,\O_X) = H^1(X,\O_X)_{ss} \oplus H^1(X,\O_X)_n$ induced by the $p$-linear morphism $F \colon H^1(X,\O_X) \to H^1(X,\O_X)$. By assumption, there are no \'etale Galois covers of $X$ for $\Z / p \Z$. By \cite[\S16]{Se58} (see also \cite[p.322]{MOagII}), this implies that $H^1(X,\O_X)_{ss}=0$.

We now want to compare $H^1(X,\O_X)$ with $H^1(X_y,\O_{X_y})$ for fibers $X_y:=h^*(y)$, where $y \in \P^1$. First, we have the isomorphism $$\alpha_1 \colon H^1(X,\O_X) \to H^0(R^1h_*\O_X)$$ given by the Leray spectral sequence. We remark that $\alpha_1$ is induced by $H^1(X,\O_X) \to H^1(h^{-1}(U),\O_{h^{-1}(U)}) \to R^1h_*\O_X(U)$ for $U \subset \P^1$ open affine sets. For each $y \in \P^1$, we have $\alpha_2 \colon H^0(R^1h_*\O_X) \to R^1h_*\O_X \otimes k(y)$ given by restriction to $y \in \P^1$, where $k(y)\simeq \K$ is the residue field. Finally, we have the natural isomorphism $\alpha_3 \colon R^1h_*\O_X \otimes k(y) \to H^1(X_y,\O_{X_y})$ as in Grauert's theorem; see \cite[p.288]{Hart77}. The composition $\alpha_3 \circ \alpha_2 \circ \alpha_1 \colon H^1(X,\O_X) \to H^1(X_y,\O_{X_y})$ is simply given by the restriction, and so we have $$\alpha_3 \circ \alpha_2 \circ \alpha_1 \circ F(\gamma)= F_y \circ \alpha_3 \circ \alpha_2 \circ \alpha_1(\gamma),$$ where $F_y \colon H^1(X_y, \O_{X_y}) \to H^1(X_y, \O_{X_y})$ is the $p$-linear morphism corresponding to the Frobenius action on $X_y$.

We now suppose that some $a_i \geq 0$, and so there is a nonzero class $\gamma \in H^1(X,\O_X) \cong H^0(R^1h_*\O_X)$. Then there is an open set $U \subset \P^1$ such that $\alpha_2(\gamma) \in R^1h_*\O_X \otimes k(y)$ is not zero, for all $y \in U$; cf. \cite{Ray94}. Therefore, $\alpha_3 \circ \alpha_2 \circ \alpha_1(\gamma) \in H^1(X_y,\O_{X_y})$ is not zero for all $y \in U$. But since $h$ has an ordinary fiber, there is a nonempty open set $V$ of $\P^1$ such that $X_y$ is ordinary for $y \in V$. But then, for $y \in V \cap U$, we have a nonzero class $\alpha_3 \circ \alpha_2 \circ \alpha_1(\gamma) \in H^1(X_y,\O_{X_y})$ which is nilpotent under Frobenius, and this is a contradiction. Therefore we have $a_i<0$ for all $i$, and so $H^1(X,\O_X)=0$.
\end{proof}

%----------------------------------------------------------------------------------------------------------------------------------------------
\subsection{Density with reduced $\pic^0$} \label{s53}

We fix an algebraically closed field $\K$ of characteristic $p>0$.

\begin{proposition}
Given real numbers $\alpha >0$, and $0<\epsilon <<1$, there are integers $e,r,l >>0$ such that $$\Big| \frac{p^rle(2l-5)}{(2l-2)(el(2l-1)-2)+lep^r} - \alpha \Big| < \epsilon.$$

\label{limit}
\end{proposition}

\begin{proof}
Let $l:=xp^y$ and $r:=z+y$. First we want to show that for $x>>0$ there are $z>0$ such that $|\frac{p^z}{2x} - \alpha| \leq \frac{\epsilon}{3}$. If $\beta=2\alpha - \frac{2\epsilon}{3}$ and $\gamma=2\alpha + \frac{2\epsilon}{3}$, then we want to show that for $x>>0$, there are $z>0$ so that $$x \beta \leq p^z \leq x \gamma $$ for a given pair $0<\beta<\gamma$. For this, it is enough to show that there is a positive integer $x_0$ such that for all $x \geq x_0$ we have $(x+1)\beta < x\gamma$. This is true because of $\lim_{x\to \infty} x/(x+1)=1$. So let us fix such $x>>0$ and $z>0$. We now fix $y>>0$ such that $$\Big| \frac{p^r(2l-5))}{4l^2-6l+2+p^r} - \frac{p^z}{2x} \Big| < \frac{\epsilon}{3}.$$ Finally, we take $e>>0$ such that $$\Big|\frac{p^rle(2l-5)}{(2l-2)(el(2l-1)-2)+lep^r} - \frac{p^r(2l-5)}{4l^2-6l+2+p^r} \Big| < \frac{\epsilon}{3}.$$ Then the claim is proved via the triangle inequality.
\end{proof}

Let $\alpha >0$, and $0< \epsilon < <1$. We fix integers $e,r,l >>0$ as in the previous proposition. Let $d=2l$, and consider the arrangement ${\AR}'_0 := {\AR}_0 \setminus \{S_{d+1} \}$ in the Hirzebruch surface $\FF_{ep^r}$, where ${\AR}_0$ and $S_{d+1}$ are as in \S \ref{s1}. This is, $$\AR'_0 = \{S_1,\ldots, S_d, F_1,\ldots, F_{\delta}\} $$ where $S_i \sim |S_{d+1} + ep^r F|$ with $F=\P^1$, $F^2=0$ and $S_{d+1}^2=-ep^r$, $S_i$ with $S_j$ intersect at $e$ points with multiplicity $p^r$ each, and each fiber $F_i$ passes through each of these intersection points, so $\delta=el(2l-1)$.

We can and do assume that there is a fiber $F_0$, distinct to the $F_1,\ldots,F_{\delta}$, such that $F_0$ intersected with the sections $S_1,\ldots,S_{2l}$ forms a collection of points which belongs to the (nonempty) open set of Lemma \ref{ordinary}.

Let $q>p$ be a large enough prime, so that we have partitions (by positive integers) \begin{equation} \label{equ} a_1+\ldots+a_l =q \ \ \ \text{and} \ \ \ y_1+\ldots+y_{\delta} =q.\end{equation} We consider the linear equivalences in $\FF_{ep^r}$ given by $$a_1 S_1 + (q-a_1) S_2+ \ldots + a_l S_{d-1}+ (q-a_l) S_d \sim ql \big(S_{d+1}+ep^r F \big),$$ and $y_1 F_1 + \ldots + y_{\delta} F_{\delta} \sim q F$. Hence, if $\L' := \O_{\FF_{ep^r}}(lS_{d+1} +(lep^r+1)F)$, we obtain $$\O_{\FF_{ep^r}}\Big(\sum_{i=1}^l a_i S_{2i-1} + \sum_{i=1}^l (q-a_i) S_{2i} + \sum_{i=1}^{\delta} y_i F_i \Big) \simeq {\L'}^q.$$

Let $\sigma \colon Y \to \FF_{ep^r}$ be the minimal log resolution of $\AR'_0$. In \cite[\S8]{Urz11} we prove that we can assign solutions of the above equations \eqref{equ} to the corresponding curves in $\AR'_0$, so that the pair $(\FF_{ep^r},\AR'_0)$ is asymptotic (as in \S \ref{s3}), \underline{except} for the $el$ points of intersection between $S_{2i-1}$ and $S_{2i}$ for $1 \leq i \leq l$ (below we will add their contribution to Chern numbers). Here $\L:= \sigma^*(\L')$, and $$D:=\sum_j \nu_j D_j:= \sigma^* \Big(\sum_{i=1}^l a_i S_{2i-1} + \sum_{i=1}^l (q-a_i) S_{2i} + \sum_{i=1}^{\delta} y_i F_i \Big).$$  As before, using $\L^q \simeq \O_Y(D)$, we have the $q$-th root cover $f=f_2 \circ f_3 \colon X \to \overline{W} \to Y$ where $f_3 \colon X \to \overline{W}$ is the resolution of singularities of the normal projective surface $\overline{W}$ constructed in \S \ref{s3} via the finite morphism $f_2 \colon \overline{W} \to Y$.

For the points over the intersection between $S_{2i-1}$ and $S_{2i}$, for $1 \leq i \leq l$, we will have $elp^r$ nodes in $Y$ which become $elp^r$ rational double point singularities of type $A_{q-1}$ in $ \overline{W}$. Indeed, if we consider one $F_j$ passing through one of the points in $S_{2i-1} \cap S_{2i}$, and if $G_{k,j}$ are the exceptional curves over $F_j \cap S_{2i-1} \cap S_{2i}$ for $\sigma$, then $$\sigma^*(y_jF_j+ a_i S_{2i-1} + (q-a_i) S_{2i})= y_jF_j+ a_i S_{2i-1} + (q-a_i) S_{2i} + \sum_{k=1}^{ep^r} (qk+y_j)G_{k,j}.$$ These singularities $A_{q-1}$ contribute to the Chern numbers as (see formulas at the end of \S \ref{s2}): $c(q-1,q)=2-\frac{2}{q}$ and $l(q-1,q)=q-1$.

Therefore, asymptotically the Chern ratio of $X$ is only affected on $c_2$, and so when $q>>0$ and multiplicities from equations \eqref{equ} are ``randomly chosen", we obtain $$\text{lim}_{q \to \infty} \frac{c_1(X)}{c_2(X)} = \frac{\bar{c}_1^2(\FF_{ep^r},\AR'_0)}{\bar{c}_2(\FF_{ep^r},\AR'_0) + dep^r} = 2 + \frac{p^rle(2l-5)}{(2l-2)(el(2l-1)-2)+lep^r}$$ which is arbitrarily close to $2+\alpha$ by Proposition \ref{limit}.

\begin{proposition}
We have that $K_{\overline{W}}$ is nef.
\label{nef2}
\end{proposition}

\begin{proof}
We know that $qK_{{\overline{W}}} \equiv f_2^*\Big(qK_Y+(q-1) D_{\text{red}} \Big).$ Let $P_1, \ldots, P_{\delta}$ be the points in $\FF_{ep^r}$ where pairs of sections in $\{S_1,\ldots,S_d \}$ intersect. Let $G_{i,j}$ be the exceptional curves over $P_i$ for $\sigma$, and let $E_i := \sum_{j=1}^{p^r} j G_{i,j}$. In this way $$K_Y \equiv -2 S_{d+1} -(2+ep^r) F + \sum_{i=1}^{\delta} E_i $$ where $F$ is the general fiber of $\pi \circ \sigma \colon Y \to \FF_{ep^r} \to \P^1$. We recall that we always use the same letter for the strict transform of a curve.

Note that $$D_{\text{red}} \equiv \sum_{i=1}^{d} S_i + \sum_{i=1}^{\delta} F_i + \sum_{i=1}^{\delta} \sum_{j=1}^{p^r} G_{i,j} \equiv \sum_{i=1}^{d} S_i + \delta F,$$ and $\sum_{i=1}^d S_i \equiv d S_{d+1} + d e p^r F - 2 \sum_{i=1}^{\delta} E_i .$ Hence $$qK_Y + (q-1) D_{\text{red}} \equiv \big((q-1)l-2q \big) S_{d+1} + \frac{q-1}{2} \sum_{i=1}^d S_i + \sum_{i=1}^{\delta} E_i + t F $$ where $t:= (l-1)qep^r-2q-dep^r+\delta (q-1)>0$ for $q>>0$. In this way, the divisor $qK_{{\overline{W}}}$ has $\Q$-effective support on $\overline{S_i}:=f_2^{-1}(S_i)$, $\overline{S_{d+1}}:=f_2^{-1}(S_{d+1})$, $\overline{G_{i,j}} := f_2^{-1}(G_{i,j})$, and $\overline{F}:= f_2^{-1}(F)$ for all $i,j$. We have the pullback relations $q\overline{S_i}\equiv f_2^*(S_i)$, $\overline{S_{d+1}}\equiv f_2^*(S_{d+1})$, $q\overline{G_{i,j}} \equiv f_2^*(G_{i,j})$, and $\overline{F} \equiv f_2^*(F)$.

To prove that $K_{\overline{W}}$ is nef, it is enough to show that $K_{\overline{W}} \cdot \Gamma \geq 0$ for $\Gamma \in \{ \overline{S_i},\overline{G_{i,j}},\overline {S_{d+1}}, \overline{F} \}$. It is then enough to intersect $qK_Y + (q-1) D_{\text{red}}$ with curves in $\{ S_i,G_{i,j},S_{d+1},F \}$. We get

\begin{itemize}
\item $\big( qK_Y + (q-1) D_{\text{red}} \big) \cdot S_{d+1}= q(ep^r-2)+\delta(q-1) >0$,

\item $\big( qK_Y + (q-1) D_{\text{red}} \big) \cdot S_i = \frac{q-1}{2}(ep^r - (d-1)ep^r) + t + \delta p^r >0$ for $1 \leq i \leq d$,

\item $\big( qK_Y + (q-1) D_{\text{red}} \big) \cdot F >0$ since $F^2=0$ and intersects curves curves in the support,

\item $\big( qK_Y + (q-1) D_{\text{red}} \big) \cdot G_{i,1} =t >0$,

\item $\big( qK_Y + (q-1) D_{\text{red}} \big) \cdot G_{i,j} =0$ for $1<j<p^r$,

\item $\big( qK_Y + (q-1) D_{\text{red}} \big) \cdot G_{i,p^r} =q-2>0$,

\end{itemize}

for $q >>0$.

\end{proof}

We now observe that we can apply asymptotic minimality to $X$ (Theorem \ref{asymptoticallyminimal}) since the contributions of $c(q-1,q)$ is $2 - \frac{2}{q}$ (see proof of Theorem \ref{asymptoticallyminimal}).

\begin{theorem}
Let $\K$ be an algebraically closed field of characteristic $p>0$. Then, for any real number $x \geq 2$, there are minimal surfaces of general type $X'$ over $\K$ such that
\begin{itemize}
\item[a)] $c_1^2(X')>0, c_2(X')>0$,

\item[b)] $\pi_1^{\text{\'et}}(X')$ is trivial,

\item[c)] $H^1(X',\O_{X'})=0$,

\item[d)] and $c_1^2(X')/c_2(X')$ is arbitrarily close to $x$.
\end{itemize}
\label{main2}
\end{theorem}

\begin{proof}
Let $x=2 +\alpha$, and let $X'$ be the minimal model of the surfaces $X$ constructed above. We already have $a)$ and $d)$. To prove $b)$ and $c)$, we consider the fibration $g \colon X \to \P^1$ induced by $\pi \circ \sigma \colon Y \to \FF_{ep^r} \to \P^1$. Notice that $g$ has sections given by the pre-images of the $S_i$, $1 \leq i \leq d$. Also, the fibers of $g$ over the fibers $F_i$ of $\FF_{ep^r} \to \P^1$ give nonmultiple fibers which consists of trees of $\P^1$'s. So, by the Appendix, the surface $X$ is \'etale simply connected, and so is $X'$. Moreover, since the pre-image of $F_0$ is an ordinary fiber of $g$, by Theorem \ref{irregularity}, we have that $H^1(X,\O_X)=0$, and so $H^1(X',\O_{X'})=0$.
\end{proof}

\begin{remark}
In the previous theorem, we compute $H^1(X',\O_{X'})=0$ via Theorem \ref{irregularity} using that $X$ has a fibration to $\P^1$ with a simply connected fiber, a section and an ordinary fiber. If Theorem \ref{irregularity} is true for fibrations over arbitrary curves $C$ with a simply connected fiber, a section and an ordinary fiber, then it may be possible to modify the construction given in this section to prove Theorem \ref{main} with reduced Picard scheme $\pic^0(X')$ isomorphic to Jac$(C)$. Finally, if in addition we can drop the assumption on ordinary fiber (or make it milder) in Theorem \ref{irregularity}, then we could prove Theorem \ref{main} with reduced Picard scheme $\pic^0(X')$ isomorphic to Jac$(C)$ with no necessity of most of the work done in this section. A Theorem \ref{irregularity} with no such assumption could also give an approach to prove Kawamata-Viehweg vanishing theorem for rational surfaces in positive characteristic, even if we only have it for $C=\P^1$. This will be part of a future work.
\end{remark}

%------------------------------------------------------------------------------------------------------------------------------------------------
\section*{Appendix. \'Etale fundamental group of a fibration \\ By Rodrigo Codorniu} \label{ap}

The main purpose of this section is to give a tool to compute the \'etale fundamental group $\pi_1^{\text{\'et}}$ of certain fibered surface, similar to the one we have over $\C$ for the topological fundamental group $\pi_1^{\text{top}}$; see  \cite[p.311]{Nori83}, \cite[p.600]{Xiao91}. This tool may be well known to specialists. For recent general developments see \cite{Mit15}. A particular case over the complex numbers is the following.

\begin{lemma}
Let $X$ (resp. $C$) be a nonsingular projective surface (resp. curve) over $\C$. Let $h \colon X \to C$ be a surjective morphism such that $h_* \O_X = \O_C$. Assume that $h$ has no multiple fibers. Then, for every $x \in X$, we have the exact sequence  \begin{center}
$\pi_1^{\text{top}}(h^{-1}(h(x)),x) \rightarrow \pi_1^{\text{top}}(X,x) \rightarrow \pi_1^{\text{top}}(C,h(x)) \rightarrow 1$
\end{center} \label{xiao}
\end{lemma}

\begin{proof}
See e.g. \cite{Xiao91}.
\end{proof}

In \cite[Exp.X Cor.1.4]{Gro71}, we have the \'etale analogue of Lemma \ref{xiao} when $h$ is assumed to be separable \cite[Exp.X Def.1.1]{Gro71}. In our situation this means that all fibers are reduced. We will prove it assuming that $h$ has no multiple fibers. For that, we slightly modify the argument given in \cite[Ch.VI]{Mu67}. The key is to give a proof of \cite[Exp.X Prop.1.2]{Gro71} (or \cite[Th.6.2.1]{Mu67}) in our setting.

We follow definitions and notations as in \cite{Mu67} but emphasizing our particular interest. Let us fix an algebraically closed field $\K$. For a locally noetherian connected scheme $S$, let $(\text{Et}/S)$ be the category of \'etale coverings of $S$. Given a $\K$-point $s \in S$, one has the covariant fiber functor $$ F \colon (\text{Et}/S) \rightarrow (\text{finite sets}), \ \ F(X \rightarrow S)= \{ \K \text{-points of } X \ \text{over} \ s \},$$ which, roughly speaking, defines the \'etale fundamental group $\pi_1^{\text{\'et}}(S,s)$ of $S$ at $s$ as the group of automorphisms of $F$; see \cite{Gro71}, \cite{Mu67}, or \cite{Sza09} for a precise definition. Also, this group can be seen as the inverse limit of the automorphism groups of the ``Galois \'etale coverings'' of $S$. Hence the \'etale fundamental group is profinite.

Recall that any morphism of schemes $f \colon Y \rightarrow X$ induces a morphism of fundamental groups $f_{*} \colon \pi_{1}^{\text{\'et}}(Y,y) \rightarrow \pi_{1}^{\text{\'et}}(X,f(y))$ through essentially base change of \'etale coverings of $X$. All morphism of \'etale fundamental groups in this section will be of this type.

The following is a well-known fact.

\begin{proposition}
Let $T$ be tree of projective lines over $\K$. This is, $T_{\text{red}}$ is a connected nodal curve with a finite number of components, each isomorphic to $\P_{\K}^1$, and whose dual graph is a tree. Then its \'etale fundamental group is trivial.
\label{tree}
\end{proposition}

Let $X$ (resp. $C$) be a nonsingular projective surface (resp. curve) over $\K$. A \textit{fibration} is a surjective morphism $h \colon X \to C$ such that $h_*\O_X = \O_C$. We denote by $X_c$ the fiber $\spec \K \times_C X$ of $h$ over a $\K$-point $c$. If $h^*(c)=\sum a_i \Gamma_i$ is its decomposition into prime divisors $\Gamma_i$, the multiplicity of $X_c$ is defined as gcd$(\{a_i\}_i)$. When gcd$(\{a_i\}_i)>1$, we say that $X_c$ is a \textit{multiple fiber} of $h$.

\begin{lemma}
Let $h \colon X \to C$ be a fibration without multiple fibers. Let $x$ be a $\K$-point of $X$. Then, we have the exact sequence \begin{center} $\pi_1^{\text{\'et}}(X_{h(x)},x) \stackrel{\rm \varphi}
{\rightarrow} \pi_1^{\text{\'et}}(X,x) \stackrel{\rm \psi} {\rightarrow} \pi_1^{\text{\'et}}(C,h(x)) \rightarrow 1$ \end{center} \label{lemma}
\end{lemma}

\begin{proof}
Following the arguments in \cite[Ch.VI Th.6.3.2.1]{Mu67} the proof is divided in three parts. The third part involves the assumption on multiple fibers.

\textbf{(a)} ($\psi$ is surjective) It is enough to prove that for every connected \'etale covering $f \colon C^{\prime} \to C$ the base change $f^{\prime} \colon C^{\prime} \times_{C} X \to X$ is also a connected \'etale covering. It only uses that $h$ is flat and proper, the proof is identical to \cite[Ch.VI Th.6.3.2.1(a)]{Mu67}, so we omit it.

%Let $X^{\prime}=C^{\prime} \times_{C} X$, we have a commutative diagram $$ \xymatrix{    X^{\prime} \ar[r]^{f^{\prime}} \ar[d]_{h^{\prime}} & X \ar[d]^{h} \\ C^{\prime} \ar[r]_{f} & C      }$$ As $h$ is proper and $f$ is flat because is \'etale, using \cite[Ch. VI, Th. 6.1.2]{Mu67} with $n=0$ we have an isomorphism $h_{*}(\Strsh{X}) \otimes_{\Strsh{C}} \Strsh{C^{\prime}} = h^{\prime}_{*}((f^{\prime})^{*}(\Strsh{X}))$, the left hand side is isomorphic to $\Strsh{C^{\prime}}$ and the right hand is isomorphic to $h^{\prime}_{*}(\Strsh{X^{\prime}})$, hence $h^{\prime}$ is a fibration. Finally, we have $H^{0}(X^{\prime},\Strsh{X^{\prime}}) \simeq  H^{0}(C^{\prime},\Strsh{C^{\prime}}),$ and if $X^{\prime}$ were disconnected, the cohomology group of the left had side would have at least two direct summands, which contradicts the connectedness of $C^{\prime}$.

\textbf{(b)}($\psi \circ \phi$ is trivial) Same as \cite[Ch.VI Th.6.3.2.1(b)]{Mu67}.

 %In this case one has prove that: given $f \colon C^{\prime} \to C$ \'etale covering of $C$, then $X_{c}^{\prime}=C^{\prime} \times_{X} X_{c}$ is a trivial cover of $X_{c}$. This is in fact not difficult to see.

\textbf{(c)}($\textrm{Ker} \, \psi \subset \textrm{Im} \, \phi$) Now we need to prove that for any connected \'etale covering $f \colon X^{\prime} \rightarrow X$ such that $X_{h(x)}^{\prime}= X_{h(x)} \times_{X} X^{\prime} \to X_{h(x)}$ has a section over $X_{h(x)}$, there exist a connected \'etale covering $C^{\prime} \to C$ such that the natural map $X^{\prime} \to C^{\prime} \times_{C} X$ is an isomorphism.

We have that $h \circ f$ is proper, and so we can apply the Stein factorization (see \cite[\S 6.2]{Mu67}) $$\xymatrix{ X^{\prime} \ar[r]^{f} \ar[d]_{h^{\prime}} & X \ar[d]_{h} \\ C^{\prime} \ar[r]_{f^{\prime}}& C}$$ where $h^{\prime}$ is a fibration, and $f^{\prime}$ is a finite morphism. The key will be to prove that $f^{\prime}$ is a connected \'etale covering.

Because $f$ is a connected \'etale morphism and $X$ is a nonsingular projective surface, $X^{\prime}$ is also a nonsingular projective surface (see \cite[Prop. 5.2.12]{Sza09}). Furthermore, it is easy to see that $C^{\prime}$ is a nonsingular projective curve. We also have that $f^{\prime}$ is flat, and the extension of fields of rational functions $K(C^{\prime})/K(C)$ is separable. On the contrary, it would have a purely inseparable part, and this would correspond to a factorization of $f^{\prime} \colon C^{\prime} \to C$ into a separable and purely inseparable part, the latter a composition of $\K$-linear Frobenius morphisms; see \cite[Ch.IV \S2]{Hart77}. Then this morphism $f^{\prime}$ would be ramified everywhere, but that would give multiple fibers to $h^{\prime}$. This contradicts the assumption on fibers for $h$, since $f$ is an \'etale covering, and so unramified. This simple observation is the key to prove that the map $f^{\prime}$ is actually unramified, and so \'etale.

Let us suppose the contrary, so there exists a point $c \in C$ with ramification, hence $(f^{\prime})^{*}(c)= \sum_{i} n_{i} p_{i}$
for some points $p_{i} \in C^{\prime}$, and suppose $n_{1}>1$. We have $ h^{*}(c)= \sum_{i=1}^{t} m_{i} \Gamma_{i}$ with gcd$(m_1,\ldots,m_t)=1$ by assumption, and we write $$ f^{*}(h^{*}(c))= \sum_{i=1}^{t} m_{i} \Big( \sum_{j=1}^{k(i)} \Gamma_{i,j} \Big)$$ where $\Gamma_{i,j}$ are all the distinct curves in $f^{*}(\Gamma_{i})$.

Let us write $(h^{\prime})^{*}(p_{1}) = \sum_{l} a_{1,l} \Gamma^{\prime}_{1,l}$ for some $a_{1,l} \in \mathbb{Z}_{>0}$, and some curves $\Gamma^{\prime}_{1,l}$. Now, by using that $h^{-1}(c)$ is connected with finitely many components, one can easily prove that for every index $i\in \{1,2,\ldots,t \}$, there exists $l$ such that $\Gamma^{\prime}_{1,l} \in f^{*}(\Gamma_{i})$. Then the commutative diagram above implies that $n_1$ divides $m_i$ for every $i$, which is not possible. Therefore $f^{\prime}$ is unramified, hence a connected \'etale covering.

Finally, it remains to prove that $X^{\prime} \simeq C^{\prime} \times_{C} X$ by the natural map, under the assumption that $X_{h(x)}^{\prime} \to X_{h(x)}$ has a section over $X_{h(x)}$. That is done in \cite[Ch.VI Th.6.3.2.1(c)]{Mu67} without applying the hypothesis (they have) on separability of the morphism.

\end{proof}

The assumption ``fibration without multiple fibers" is the minimum requirement to have the exact homotopy sequence of Lemma \ref{lemma}. Indeed, one can construct fibrations with multiple fibers where the sequence is not exact. See \cite[\S 1]{Xiao91} for examples over $\C$ for $\pi_1^{\text{top}}$. In that case, to have an exact homotopy sequence, one has to replace $\pi_1^{\text{top}}$ of the base curve by the $\pi_1^{\text{top}}$ of a suitable orbifold, depending on the multiple fibers. Then, to have an \'etale example, one can use that the profinite completion of $\pi_1^{\text{top}}$ is $\pi_1^{\text{\'et}}$.

\begin{corollary}
Let $h \colon X \rightarrow C$ be a fibration which has a section and a simply connected fiber, and let $x \in X$ be a $\K$-point. Then $$\pi_1^{\text{\'et}}(X,x) \simeq \pi_1^{\text{\'et}}(C,h(x)).$$ \label{corlemma}
\end{corollary}

\begin{proof}
The section implies no multiple fibers. Hence, the isomorphism is a trivial consequence of the preceding lemma.
\end{proof}
%----------------------------------------------------------------------------------------------------------------------------------------------

\end{document}